\documentclass[12pt,twoside]{article}

\usepackage{amsfonts}
\usepackage{amsmath}
\usepackage{graphicx}

\date{}

\begin{document}

\centerline{}

\centerline{}

\centerline {\Large{\bf A Generalized result of Output Stabilizability }}

\centerline {\Large{\bf }}

\centerline{}

\centerline{\bf {Faouzi Haddouchi}}

\centerline{}

\centerline{Department of Physics, University of Sciences and Technology}

\centerline{El M'naouar, BP 1505, Oran, 31000, Algeria}

\centerline{fhaddouchi@gmail.com}

\newtheorem{Theorem}{\quad Theorem}[section]
\newtheorem{Definition}[Theorem]{\quad Definition}
\newtheorem{Corollary}[Theorem]{\quad Corollary}
\newtheorem{Proposition}[Theorem]{Proposition}
\newtheorem{Lemma}[Theorem]{\quad Lemma}
\newtheorem{Example}[Theorem]{\quad Example}
\newtheorem{Remark}[Theorem]{Remark}
\newenvironment{proof}[1][Proof]{\noindent\textbf{#1:} }

\begin{abstract}
Output stabilizability of a class of infinite dimensional linear systems is studied in this paper. A criterion for the system to be output stabilizable by a linear bounded feedback $u=Fx$, $F\in L(Z,\mathbb{R}^{^{p}})$  will be given.
\end{abstract}



\section{Introduction}

In this note, inspired by the result in [2]
for output stabilizability of the diffusion equation, we proposed a
new output stabilizability criterion for a class of infinite dimensional linear systems with
multi-actuators and multi-sensors.
The system we consider is described by the abstract differential equation

\begin{equation}
\mathcal{(S)} \left\{
\begin{array}{l}
\overset{.}{x}=Ax+Bu \\
{x}(0)={x}_{0}%
\end{array}%
\right.
\end{equation}
where $ A$ generates a strongly continuous
semigroup $(S(t))_{t\geq 0}$ on $ Z$ (state space); $U$ is the control space and the control function $u(.)\in L^{2}(0,T;U)$; $B\in L(U,Z)$; $U$ and $ Z$ are supposed to be a separable Hilbert spaces.
The system (1) is augmented by the output equation
\begin{equation}
\mathcal{(E)}\  \ y=Cx
\end{equation}
where $C\in L(Z,Y)$, $Y$ is the observation (output) space, a separable Hilbert space, $y(.)\in L^{2}(0,T;Y)$.

The system we shall characterize its output stabilizability is assumed to be controlled via $p$ actuators $(\Omega _{i},g_{i})_{_{1\leq i\leq p}}$ and takes the form
\begin{equation}
\frac{\partial z}{\partial t}(\xi,t)=\Delta z+kz+\overset{p}{%
\underset{i=1}{\sum }}g_{i}(\xi)u_{i}(t)  \ \ \ in \ \Omega \times (0,T),
\end{equation}
with boundary conditions
\begin{equation}
z(\xi,t)=0 \ \ in  \ \partial {\small \Omega \times (0,T)},
\end{equation}
and the initial condition
\begin{equation}
z(\xi,0)=z_{0}(\xi) \ \ in \ \Omega ,
\end{equation}
with the output function given by
\begin{equation}
y(t)=\left[
\begin{array}{c}
y_{1}(t) \\
. \\
. \\
y_{q}(t)%
\end{array}%
\right]
\end{equation}
where%
\begin{equation}
\ y_{i}(t)=\int_{_{D_{i}}}f_{i}(\xi)z(\xi,t)d\xi,
\end{equation}

and $\Delta$ is the Laplacian operator, $\Omega $ is bounded and open in $ \mathbb{R}^{n}$ with smooth boundary $\partial \Omega $, $g_{i}\in L^{2}(\Omega _{i})$%
, $\Omega _{i}\subset \Omega $, $\Omega _{i}\cap \Omega _{j}=\emptyset$, $k>0$ and $\left( D_{i},f_{i}\right) _{1\leq i\leq
\text{ }q}$ is a suite of sensors with $D_{i}\subset \Omega $ and $f_{i}\in
L^{2}(D_{i}).$\\
The above system (3)-(7) is a special form of (1)-(2) where
$Z=L^{2}(\Omega )$, $A=\Delta+kI$,
$\ D(A)=H_{0}^{1}(\Omega
)\cap H^{2}(\Omega )$ and
\begin{equation}
B\left[
\begin{array}{c}
u_{1} \\
. \\
. \\
u_{p}%
\end{array}%
\right] =\overset{p}{\underset{i=1}{\sum }}g_{i}u_{i}.
\end{equation}
Clearly the output function (6) may be written in the form
\begin{equation}
 y(t)=Cx(t).
\end{equation}
If the associated eigenfunctions are $\varphi _{nj}$ then
\begin{equation}
S(t)x=\underset{n=1}{\overset{\infty }{\sum }}\exp (\mu _{n}t)\overset{r_{n}%
}{\underset{j=1}{\sum }}\langle x,\varphi _{nj}\rangle \varphi _{nj}
\end{equation}
where $r_{n}$ is the multiplicity of the eigenvalue $ \mu _{n}$ .

In this work the case when the eigenvalue are $ \mu _{n}$ with multiplicity $r_{n}$ is traited. Our results extend and complete those established in [2].

This paper is organized as follows: We recall in section 2, the notions
of approximate controllability, state and output stabilizability for infinite dimensional systems defined in Hilbert spaces.

In section 3, we give a generalization of the results presented in [2].

\section{Preliminaries}
We consider the system $\mathcal{(S)}$ augmented by the output equation $\mathcal{(E)}$ defined
respectively by (1) and (2).

\begin{Definition}
We say that the system $\mathcal{(S)}$(or the pair$\ (A$,$B)$) is approximately controllable
if $\ \mathcal{N}=\left\{ 0\right\}$.

Where $\mathcal{N}=\underset{{\small t\geq 0}}{\bigcap }\ker B^{\ast }S^{\ast
}\left( t\right) \ $.

$ \mathcal{L}=\mathcal{N}^{\perp }$ and$\ \mathcal{N}$ are called, the controllable and
uncontrollable subspaces of the system $\mathcal{(S)}$, respectively.

\end{Definition}

According to [4], we can decompose the state space $Z$ as $\mathcal{L}\oplus\mathcal{N}$ and then the system (1)-(2) can be written as:%
\begin{equation}
\left\{
\begin{array}{l}
\overset{.}{x_{_{1}}}=A_{_{11}}x_{_{1}}+B_{_{1}}u\quad \\
\overset{.}{x_{_{2}}}=A_{_{22}}x_{_{2}}\quad \quad \quad \quad \\
y=y_{_{1}}+y_{_{2}}\quad%
\end{array}%
\right.
\end{equation}

\-where $y_{_{i}}=C_{_{i}}x_{_{i}}$ , for $i=1,2$.

\begin{Definition}
The system $\mathcal{(S)}$ is said to be exponentially
stabilizable if there is an $\ F\in L(Z,U)$ such that the semigroup $%
S_{_{A+BF}}\left( t\right) $ is exponentially asymptotically stable.

Where $%
S_{_{A+BF}}\left( t\right) $ is the semigroup generated by $A+BF$.
\end{Definition}

\begin{Definition}

The system $\mathcal{(S)}$ augmented by the output equation $\mathcal{(E)}$ is output stabilizable by a bounded feedback if there
is an $F\in L(Z,U)$ such that the output $y(t)$ of the closed system
\begin{equation}
\overset{.}{x}\left( t\right) =(A+BF)x(t),\text{ }x(0)=x_{_{0}}
\end{equation}

is exponentially stable, i.e., $y(t)$ converges to zero when $t\rightarrow
\infty $ , for every $x_{0}\in Z$. See e.g.,[1],[3], [4].
\end{Definition}

\section{Main Results}

We need the following lemmas in the proof of our proposition.

\begin{Lemma}
The uncontrollable subspace $\mathcal{N}$ of the system (3)-(7) is of the
following form%

\begin{equation}
\mathcal{N}=\overline{span} \left\{\overset{r_{n}}{\underset{j=1}{\text{{\small $%
\sum $}}}}{\small \alpha }_{_{j}}{\small \varphi }_{_{nj}}\  /{\small B}%
_{_{n}}^{^{\ast }}{\small v}=0,\ {\small v}=( {\small \alpha }%
_{_{1}},...,{\small \alpha }_{_{r_{n}}}) ^{^{{\small T}}}\right\}
\end{equation}

where
$B_{n}=(\langle g_{i},\varphi _{nj}\rangle _{L^{2}(\Omega _{i})}), \ 1\leq
i\leq p, \ 1\leq j\leq r_{n\text{ }}$ and\
$\overline{span}%
\left\{ e_{m}\text{, }m\in I\right\} $ denotes the closed subspace generated
by the vectors $e_{m}\text{, }m\in I$, $T$ means transpose.

\end{Lemma}

\begin{proof}
 As in the proof of Lemma 3.2 in [2], we have $B^{\ast }S^{\ast }\left( t\right) x=0$ if and only if
\begin{equation}
\langle E(\mu _{n})x,g_{i}\rangle =0, for\ all \ n\geq 1,\
i=1,..,p
\end{equation}
where
\begin{equation}
E(\mu _{n})=\overset{r_{n}}{\underset{j=1}{\sum }}\langle .,\varphi
_{nj}\rangle \varphi _{nj}
\end{equation}

Noting that it is easy to see that

\begin{equation}
J=\left\{ n\ /\ rankB_{_{n}}<r_{_{n}}\right\} =\left\{ n\ /\ \ker \
B_{_{n}}^{\ast}\neq\left\{ 0\right\} \right\} .
\end{equation}

Let $x\in E(\mu_{_{n_{0}}})\mathcal{N}$, $x\neq0$, for a certain $%
n_{_{0}}\in J$.
Then
\begin{equation}
B_{_{n_{0}}}^{\ast }v_{_{n_{0}}}=0,
\end{equation}
\ \text{with}$\ v_{_{n_{0}}}=\left(
\langle x,\ \varphi _{_{n_{0}1}}\rangle, ...,\langle x,\ \varphi
_{_{n_{0}r_{n_{0}}}}\rangle \right)  ^{^{T}}\neq 0$.

This shows that
\begin{center}
$\mathcal{N} \subset \overline{span} \left\{\overset{r_{n}}{\underset{j=1}{\text{{\small $%
\sum $}}}}{\small \alpha }_{_{j}}{\small \varphi }_{_{nj}}\  ,{\small B}%
_{_{n}}^{^{\ast }}{\small v}=0,\ {\small v}=( {\small \alpha }%
_{_{1}},...,{\small \alpha }_{_{r_{n}}}) ^{^{{\small T}}}\right\} $
\end{center}

The remaining part of the proof is easy to establish and will be
omitted here.
\end{proof}
From the previous Lemma we deduce the following
consequence

\begin{Lemma}

The controllable subspace $\mathcal{L}$ of the system (3)-(7) is given by
\begin{equation}
\mathcal{L}=\overline{span} \left\{\overset{r_{n}}{\underset{j=1}{\text{{\small $%
\sum $}}}}{\small \alpha }_{_{j}}{\small \varphi }_{_{nj}}\  /  ( {\small \alpha }%
_{_{1}},...,{\small \alpha }_{_{r_{n}}}) ^{^{{\small T}}}\in ImB_{n}\right\}
\end{equation}
where $B_{n}=(\langle g_{i},\varphi _{nj}\rangle _{L^{2}(\Omega _{i})}), \ 1\leq
i\leq p, \ 1\leq j\leq r_{n\text{ }}$.
 
\end{Lemma}
We are now in position to prove the main result of this section

\begin{Proposition}
Suppose there are $p$ actuators $(\Omega _{i},g_{i})_{_{1\leq i\leq p}}$ and $q$ sensors $\left( D_{i},f_{i}\right) _{1\leq i\leq \text{ }q}$, then the system (3)-(7) is output stabilizable if and only if

\begin{equation}
\mu _{n}<0, \ for \ all\ n \ in \ K
\end{equation}
where
\begin{equation}
K=\left\{ n / \text{ } {Im}T_{n}\neq \left\{ 0\right\} \text{ and }\ker
B_{n}^{\ast }\neq \left\{ 0\right\} \right\}
\end{equation}
and
 $B_{n}=(\langle g_{i},\varphi _{nj}\rangle _{L^{2}(\Omega _{i})}), T_{n}=(\langle f_{k},\varphi _{nj}\rangle _{L^{2}(D _{k})}),1\leq i\leq p,1\leq j\leq r_{n\text{ }}, 1\leq
k\leq q$.
\end{Proposition}

\begin{proof}
Similar to the proof of Proposition 3.4 in [2], it suffices to study the stability of the output $y_{_{2}}$ on the observable subspace $\mathcal{W}$ of the subsystem

\begin{equation}
\left\{
\begin{array}{l}
\overset{.}{x}_{2}^{1}=A_{22}^{1}x_{2}^{1} \\
\overset{.}{x}_{2}^{2}=A_{22}^{2}x_{2}^{2}\ \qquad \\
\text{\thinspace }y_{2}\text{ }=C_{2}^{2}x_{2}^{2}\qquad%
\end{array}%
\right.
\end{equation}%
where
\begin{equation}
\ A_{_{_{22}}}=\left(
\begin{array}{cc}
A_{_{22}}^{^{1}} & 0 \\
0 & A_{_{22}}^{^{2}}%
\end{array}
\right) ,C_{_{2}}=\left[
\begin{array}{cc}
0 & C_{_{2}}^{^{2}}%
\end{array}
\right],
\end{equation}

$x_{_{02}}=\left[
\begin{array}{c}
x_{02}^{1} \\
x_{02}^{2}%
\end{array}%
\right] \in \mathcal{M}\oplus \mathcal{W}$, $\mathcal{W}=\mathcal{M}^{\perp }$, $x_{2}(0)=x_{02}$, $x_{0}=\left[
\begin{array}{c}
x_{_{01}} \\
x_{_{02}}%
\end{array}%
\right] \in \mathcal{L}\oplus \mathcal{N}$ and
\begin{equation}
\mathcal{W}=\overline{span}\left\{\overset{r_{n}}{\underset{j=1}{\text{{\small $%
\sum $}}}}{\small \alpha }_{_{j}}{\small \varphi }_{_{nj}}\text{ }{\small ;v=}\left( {\small %
\alpha }_{_{1}},...,{\small \alpha }_{r_{n}}\right) ^{^{{\small T}}}{\small \in }%
\widetilde{\mathcal{V}}\text{\ }\right\}
\end{equation}

with $\ \widetilde{\mathcal{V}}$ $=ImT_{_{n}}\cap\ \ker\ B_{_{n}}^{\ast}$.

The output $y_{_{2}}$ of the subsystem (21) is given by%
\begin{equation}
{\small y}_{_{2}}\left( t\right) =\
\begin{bmatrix}
\underset{_{n\in K}}{\sum }\exp {\tiny (\mu }_{_{n}}{\tiny t)}\overset{r_{n}}%
{\underset{_{j=1}}{\sum }}{\tiny \langle x}_{_{0}}{\tiny ,\ \varphi }_{_{nj}}%
{\tiny \rangle \langle \ f}_{_{1}}{\tiny ,\varphi }_{_{nj}}{\tiny \rangle } \\
{\tiny .} \\
{\tiny .} \\
{\tiny .} \\
\underset{_{n\in K}}{\sum }\exp {\tiny (\mu }_{_{n}}{\tiny t)}\overset{r_{n}}%
{\underset{_{j=1}}{\sum }}{\tiny \langle x}_{_{0}}{\tiny ,\ \varphi }_{_{nj}}%
{\tiny \rangle \langle \ f}_{_{q}}{\tiny ,\varphi }_{_{nj}}{\tiny \rangle }%
\end{bmatrix}%
\ \ \ \ \ \ \
\end{equation}

where
\begin{equation}
K=\left\{ n /\ ImT_{_{n}}\neq \left\{ 0\right\} \ \text{and }\ker
B_{_{n}}^{\ast }\neq \left\{ 0\right\} \right\} ,
\end{equation}

The sufficient condition is straightforward. Now we shall prove the converse.
Suppose that the output $y_{_{2}}(t)$ is exponentially stable but
for a certain $n_{_{0}}\in K$ , $\mu _{_{n_{0}}}\geq 0$ , then there are
positive $M$ and $\omega $ such that%
\begin{equation}
\left\Vert \ y_{_{2}}(t)\right\Vert _{\mathbb{R}^{^{q}}}{\small \leq M}\exp
{\small (-\omega t)}\left\Vert \ x_{_{0}}\right\Vert \text{\ \ for every}\
x_{_{0}}\in Z
\end{equation}

Set $x_{_{0}}=\varphi_{_{n_{0}j}}\ $in equation (26) where\ $j\in\left\{
1,...,r_{n_{0}}\right\} $ \ ($j\ $fixed arbitrary )

Then we obtain%
\begin{equation}
\left\vert \langle f_{_{k}},\varphi _{_{n_{0}j}}\rangle \right\vert {\small %
\leq M}\exp \left\{ -(\omega +\mu _{_{n_{0}}})t\right\}  \text{\ for
all }t\geq 0,\ k=1,...,q.
\end{equation}

Thus $ImT_{_{n_{0}}}=\left\{ 0\right\} $ and this contradicts the assumption
that $n_{_{0}}\in K$.

\end{proof}

\begin{Remark}
It is noteworthy that if $p\geq \underset{n}{\sup }$ $r_{n}$ and $\ rank$ $B_{n}=r_{n}$, for all $n$,
then the approximate controllability is achieved and by virtue of Theorem
7.2 in [3], the system (3)-(7) is output stabilizable.
\end{Remark}

\end{document}